\documentclass[12pt]{amsart}
\usepackage{amssymb}
\usepackage{amsfonts}
\usepackage{latexsym}
\usepackage{amscd}
\usepackage[mathscr]{euscript}
\usepackage{xy} \xyoption{all}

\newcount\parracount
\newcount\recount
\newtheorem{definition}{Definition}
\newtheorem{theorem}[definition]{Theorem}

\newtheorem{corollary}[definition]{Corollary}

\newtheorem{example}[definition]{Example}

\newtheorem{proposition}[definition]{Proposition}

\usepackage{setspace}

\def\Z{\mathbb Z}

\begin{document}

\title[Cohn path algebras have IBN]{Cohn path algebras have Invariant Basis Number}

\author{Gene Abrams}
\address{Department of Mathematics, University of Colorado,
Colorado Springs CO 80918 U.S.A.} \email{abrams@math.uccs.edu}
 \author{M\"{u}ge Kanuni}
 \address{Department of Mathematics, D\"{u}zce University, 
Konuralp D\"{u}zce 81620 Turkey} \email{mugekanuni@duzce.edu.tr}

\thanks{The first  author is partially supported by a Simons Foundation Collaboration Grants for
Mathematicians Award \#208941. The second author is supported by a U.S. Department of State  
2012-2013 Fulbright Visiting Scholar Program Grant and by the Scientific and Technological Research Council of Turkey (T\"{U}B\.{I}TAK-B\.{I}DEB) 2219 International Post-Doctoral Research Fellowship during her sabbatical visit to the University of Colorado Colorado Springs. This author would like to thank her colleagues at the host institution for their hospitality.} 

\begin{abstract}
For any finite directed graph $E$ and any field $K$ we show that the Cohn path algebra $C_K(E)$ has the Invariant Basis Number property, moreover Invariant Matrix Number property.
\end{abstract}

\maketitle

\date{}

\maketitle



\onehalfspacing

For any directed graph $E$ and field $K$ the {\it Leavitt path algebra $L_K(E)$ of $E$ with coefficients in $K$} has been the object of intense research focus since its introduction in 2005, see e.g.  \cite{AA1} and \cite{AMP}.  Leavitt path algebras are generalizations of the {\it Leavitt algebras} $L_K(1,n)$ ($n \geq 2$) introduced by Leavitt in \cite{L1}.   A ring $R$ is said to have the {\it Invariant Basis Number} property (or more simply {\it IBN}) in case 
for any pair of positive integers $m\neq m'$ we have 
that the free left $R$-modules $R^m$ and $R^{m'}$ are not isomorphic. 
The Leavitt algebras  fail to have the IBN property: for instance, if $R = L_K(1,n)$, then ${}_RR^1 \cong {}_RR^n$.   (Indeed, the search for non-IBN rings motivated much of Leavitt's work.)   Additional examples abound of more general Leavitt path algebras which also fail to have the IBN property.

Let $M_n(R)$ denote $n \times n$ matrices over a ring $R$. We say that $R$ has the {\it Invariant Matrix Number} property, in case $M_i(R)\ncong M_j(R)$ for every pair of positive integers $i \neq j$.

Recently, the {\it Cohn path algebra} $C_K(E)$ of a directed graph $E$ has been defined and investigated, see e.g.,  \cite{AAS}, \cite{AM}, and/or  \cite{AG}.  As it turns out (see Theorem \ref{everyCohnisLeavitt}), for any directed graph $E$ there exists a directed graph $F$ for which $C_K(E) \cong L_K(F)$.    Rephrased: every Cohn path algebra is in fact also a Leavitt path algebra.
The purpose of this short note is to establish that, the abundance of non-IBN Leavitt path algebras notwithstanding, every Cohn path algebra has the Invariant Basis Number property, indeed, has the stronger Invariant Matrix Number property. 

\medskip

Throughout this note $E = (E^0, E^1, s, r)$ will denote a directed graph with vertex set $E^0$, edge set $E^1$, source function $s$, and range function $r$.    In particular, the source vertex of an edge $e$ is denoted by $s(e)$, and the range vertex by $r(e)$. We will assume that $E$ is {\it finite}, i.e. that both $E^0$ and $E^1$ are finite sets.  A {\it sink} is a vertex $v$ for which the set $s^{-1}(v) = \{e\in E^1 \mid s(e) = v\}$ is empty; any non-sink is called a  {\it regular} vertex.

For any directed graph $E$, we denote by $A_E$ the incidence matrix of $E$. Formally, if $E^0 = \{v_i \mid 1 \leq i \leq n\}$, then  $A_E = (a_{i,j})$ is the $n \times n$ matrix for which $a_{i,j}$ is the number of edges $e$ having $s(e) = v_i$ and $r(e)=v_j$.   In particular, if $v_i \in E^0$ is a sink, then $a_{i,j} = 0$ for all $1\leq j \leq n$, i.e., the $i^{th}$ row of $A_E$ consists of all zeros.   We assume throughout  that the vertices $\{v_i \mid 1 \leq i \leq n\}$ have been labelled in such a way that the vertices  $\{v_i \mid 1 \leq i \leq t\}$  are the regular vertices, and $\{v_i \mid t+1 \leq i \leq n\}$  are the sinks (if any).

\begin{definition}\label{definition}  
{\rm Let $K$ be a field, and let $E$ be a  graph. The {\em Cohn path $K$-algebra} $C_K(E)$ {\em of $E$ with coefficients in $K$} is  the $K$-algebra generated by a set $\{v\mid v\in E^0\}$, together with a set of variables $\{e,e^*\mid e\in E^1\}$, which satisfy the following relations:

(V)  $vw = \delta_{v,w}v$ for all $v,w\in E^0$, \  

  (E1) $s(e)e=er(e)=e$ for all $e\in E^1$,

(E2)  $r(e)e^*=e^*s(e)=e^*$ for all $e\in E^1$, and 

 (CK1)  $e^*e'=\delta _{e,e'}r(e)$ for all $e,e'\in E^1$.

\noindent
The {\em Leavitt path $K$-algebra} $L_K(E)$ {\em of $E$ with coefficients in $K$} is the $K$-algebra generated by the same set $\{v\mid v\in E^0\}$, together with the same set of variables $\{e,e^*\mid e\in E^1\}$, which satisfy the same set of relations (V), (E1), (E2), and (CK1), and also satisfy the additional relation

(CK2)  $v=\sum _{\{ e\in E^1\mid s(e)=v \}}ee^*$ for every regular vertex $v\in E^0$.  
\hfill $\Box$
}
\end{definition}
Specifically, if we let $N \subseteq C_K(E)$ denote the ideal of $C_K(E)$ generated by the elements of the form $v~-~\sum_{\{e\in E^1 \mid s(e)=v\}} ee^*,$ where $v \in E^0$ is a regular vertex, then  we may view the Leavitt path algebra $L_K(E)$   as the quotient  algebra 
$L_K(E) \cong C_K(E)/N.$      For a finite graph $E$ we have that both $C_K(E)$ and $L_K(E)$ are unital, each having  identity $1 = \sum_{v\in E^0}v$.    A perhaps-surprising additional connection between Cohn path algebras and Leavitt path algebras is given here; the discussion represents a specific case of a more general result described in \cite[Section 1.5]{AAS}.

 \begin{definition} \label{FgraphDef}
 {\rm Let $K$ be a field, let $E$ be an arbitrary graph
and let $Y$ denote the set  of regular vertices of $E$. Let $Y'=\{ v'\mid v\in Y \}$ be a disjoint copy of
$Y$. For $v\in Y$ and for each
edge $e$ in $E$ such that $r_E(e)=v$, we consider a new symbol $e'$.
We define the graph $F = F(E)$, as follows:
 $F^0=E^0\sqcup Y' ;  F^1=E^1\sqcup \{ e' \mid r_E(e)\in Y \}; $ and for each $e\in E^1$, $s_{F}(e)=s_E(e)$, 
 $s_{F}(e')=s_E(e)$,  $r_{F}(e)=r_E(e)$, and  $r_{F}(e')=r_E(e)'$.
 \hfill $\Box$}
\end{definition}
Less formally, the graph $F = F(E)$ is built from $E$  by adding a new vertex to $E$ corresponding to each non-sink of $E$, and then including new edges to each of these new vertices in the same configuration as their counterpart vertices in $E$.   Observe in particular that each of the new vertices $v'\in Y'$ is a  sink in $F$.    (We note that the graph $F(E)$ as defined here is the graph $E(\emptyset)$ of \cite[Definition 1.5.15]{AAS} when $E$ is finite.)  
The incidence matrix $A_F$ of $F=F(E)$ is the $(n+t) \times (n+t)$ matrix in which,  for $1 \leq i \leq t$, the $i^{th}$ row is $(a_{i,1}, a_{i,2}, ..., a_{i,n}, a_{i,1}, a_{i,2}, ... , a_{i,t})$ (where 
$A_E=(a_{ij})$ is the incidence matrix of $E$), and the remaining $n$ rows are zeroes.  

\begin{example}
\label{exam:Cohn=Lea}
{\rm  Let $E$ be the  graph
$ \xymatrix{
{\bullet}^u  \ar[r]^e & {\bullet}^v  \ar[r]^f & {\bullet}^w }
$.  
Then the graph
$F = F(E)$ is:

$ 
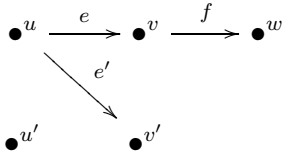
\xymatrix{
{\bullet}^u \ar[r]^e \ar[rd]^{e'} & {\bullet}^v  \ar[r]^f  & {\bullet}^w  \\ {\bullet}^{u'} &  {\bullet}^{v'}  &   }
$

\noindent
We note that there is no new vertex in $F$ corresponding to $w$, as $w$ is a sink in $E$.
\hfill $\Box$}
\end{example}
\begin{example}\label{F(R_2)Example}
{\rm Let 
$R_2$ be the graph $  \xymatrix{   \bullet^v 
\ar@(u,r)^{e} 
\ar@(d,r)_{f}
  }
$. 
Then the graph  $F = F(R_2)$ is:
 
$ 
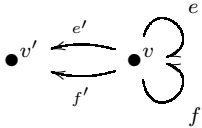
\xymatrix{\bullet^{v'} &  \bullet^v 
\ar@(u,r)^{e} 
\ar@(d,r)_{f}
\ar@/^-.5pc/[l]_{  {}_{{}_{e'}}}
\ar@/^.5pc/[l]^{   {}^{{f'}}}
  }$
\hfill $\Box$}
\end{example}
Here is the first key result we will utilize.

\begin{theorem} (A specific case of  \cite[Theorem 1.5.17]{AAS}) 
\label{everyCohnisLeavitt} Let $E$ be any graph. Then there is an isomorphism of $K$-algebras 
$C_K(E)\cong
L_K(F(E)).$
\end{theorem}

For any directed graph $E$ with $|E^0| = n$ we construct the abelian monoid $M_E$ as follows.  Consider the abelian monoid  $T = (\Z^+)^{n}$ of $n$-tuples of non-negative integers.  For each regular vertex $v_i$ ($1 \leq i \leq t$)   we let $\vec{b_i}$ denote the  vector $ (0, 0, ... , 1, 0, ... 0)$ (having $1$ in the $i$-th component) of $(\Z^+)^{n}$.  We consider the equivalence relation  $\sim_E$ in   $(\Z^+)^{n}$, generated by setting 
$$\vec{b_i} \sim_E (a_{i,1}, a_{i,2}, ..., a_{i,n})$$ for each regular vertex $v_i$.  The monoid $M_E$ consists of the equivalence classes $(\Z^+)^{n} / \sim_E$; if we denote the equivalence class of an element $\vec{a} \in (\Z^+)^{n}$ by $[\vec{a}]$, then the operation in $M_E$ is given by setting $[\vec{a}] + [\vec{a'}] = [\vec{a} + \vec{a'}]$ for $\vec{a}, \vec{a'} \in (\Z^+)^{n}$.     

For instance, if $E$ is the graph of Example \ref{exam:Cohn=Lea}, then $M_E$ is the monoid $(\Z^+)^{3}$, modulo the relation $\sim_E$  generated by setting $(1,0,0)\sim_E (0,1,0)$ and $(0,1,0) \sim_E (0,0,1)$.   It is not hard to see that $M_E \cong \Z^+$.  For the graph $F=F(E)$ of that same Example, one can show that $M_F \cong (\Z^+)^3$.    On the other hand, if $R_2$ is the graph presented in  Example \ref{F(R_2)Example}, then $M_{R_2}$ is the monoid $(\Z^+)^1$, modulo the relation generated by setting $(1) \sim_{R_2} (2)$.   In this case we see that $M_{R_2} \cong \{0,x\}$, where $x + x = x$.  (N.b.: $M_{R_2}$ is {\it not} the group $\Z_2$.)   Furthermore, for $F = F(R_2)$, one can show that $M_F$ is the monoid
\[ \{ [(0,i)]\mid i \geq 0 \} \ \sqcup \  \{ [(1,i)]\mid i \geq 0 \} \ \sqcup \ \{ [(i,0)]\mid  i \geq 2\} \ \sqcup \ 
\{ [(i,1)]\mid  i \geq 2 \} .\] 
Observe in particular that for positive integers $m \neq m'$, $(m,0) \nsim_F (m',0)$ and $(m,1) \nsim_F (m',1)$.

For any ring $R$, we denote by $\mathcal{V}(R)$ the abelian monoid of isomorphism classes of finitely generated projective left $R$-modules, with operation $\oplus$.   If $P$ is a finitely generated projective left $R$-module, we denote the element of $\mathcal{V}(R)$ which contains $P$ by $[P]$.   We recast the IBN property in $\mathcal{V}(R)$ as follows:  $R$ has IBN if and only if for every pair of distinct positive integers $m\neq m'$ we have $m[R] \neq m'[R]$ as elements of $\mathcal{V}(R)$.   
Here is the second key result we will utilize.

\begin{theorem} (\cite[Theorem 3.5]{AMP}) \label{isoofmonoids}
Let $E$ be a finite graph with vertices $\{v_i \mid 1\leq i \leq n\}$, and let $K$  be any field.  Then  the assignment $[\vec{b_i}] \mapsto [L_K(E)v_i]$  yields an isomorphism of monoids $M_E \cong \mathcal{V}(L_K(E))$.   In particular, under this isomorphism, if $\vec{\rho} = (1,1,...,1) \in (\Z^+)^{n}$, we have $[\vec{\rho}]  \mapsto [L_K(E)]$.
\end{theorem}
Consequently,
\begin{corollary} \label{showingL(E)IBN}
Let $F$ be any finite graph, and $K$ any field.  Let $\vec{\rho} = (1,1,...,1) \in M_F$.  Then $L_K(F)$ has IBN if and only if for any pair of positive integers $m\neq m'$, we have $m\vec{\rho} \nsim_F m'\vec{\rho}$.  
\end{corollary}

In the graph $E$ of Example \ref{exam:Cohn=Lea}, we have that $M_E \cong \Z^+$; moreover, in this identification, $[\vec{\rho}] \mapsto 3$.   Since $m\neq m'$ obviously gives $m\cdot 3 \neq m' \cdot 3$ in $\Z^+$, we conclude that $L_K(E)$ has IBN.   Further, in the graph $F(E)$ of that same Example, we have that $M_{F(E)} \cong (\Z^+)^3$, and in this identification we have $[\vec{\rho}] \mapsto (3,1,2)$.  As before,  $m\neq m'$ obviously gives $m(3,1,2) \neq m'(3,1,2)$ in $(\Z^+)^3$, we conclude that $L_K(F(E))$ has IBN as well.  On the other hand, if $R_2$ is the graph of Example \ref{F(R_2)Example}, then $M_{R_2} \cong \{0,x\}$, and in this identification we have $[\vec{\rho}] \mapsto x$.   Since $1x = 2x$, we conclude that $L_K(R_2)$ does not have IBN.  (We note that each of these three observations is well-known, we have included them   here only to help  clarify Corollary \ref{showingL(E)IBN}.  Indeed, $L_K(R_2) \cong L_K(1,2)$, the non-IBN Leavitt algebra for $n=2$ mentioned in the introductory paragraphs of the article.)   Finally, for the graph $F=F(R_2)$, we get 
\[m \vec{\rho}= (m,m) \sim_F \begin{cases} (\frac{m}{2},0) & \mbox{if $m$ is even} \\ (\frac{m+1}{2},1) 
& \mbox{if $m$ is odd} \end{cases}\]  
for any positive integer $m$.   So by a previous observation we get that $m\neq m'$ gives $m \vec{\rho} \nsim_F m'\vec{\rho}$. Hence $L_K(F(R_2))$ has IBN property. 
That both $L_K(F(E))$ and  $L_K(F(R_2))$ have the IBN property will also follow as a consequence of Theorem \ref{MainResult}.  

 We will use the following elementary linear algebra result.

\begin{proposition}\label{linearalgprop}
Given any non-negative integers $a_{ij}$ ($1 \leq i \leq t$, $1\leq j \leq n$), there exist
$w_1,w_2, \dots, w_{n+t}$ in $\mathbb{Q}$ which satisfy the following system of $t+1$ linear equations: 
$$\begin{array}{ccccccccccccccccc}
1 & = & w_1 &+& w_2& +& ... & +&w_n& +& w_{n+1}& +& ...&  +& w_{n+t}\\
w_1 & = & 
a_{11}w_1 &+& a_{12}w_2& +& ... & + &a_{1n}w_n& +& a_{11}w_{n+1} &+ & ...
& +& a_{1t}w_{n+t} \\
w_2 & = &  a_{21}w_1 & +& a_{22}w_2& + &... & +& a_{2n}w_n& +& a_{21}w_{n+1}& +& ... 
& +& a_{2t}w_{n+t}\\
\vdots & & &&&&&&&&&&&&\\
w_t & = & 
a_{t1}w_1 & + & a_{t2}w_2 & +& ... &+ &a_{tn}w_n& +& a_{t1}w_{n+1}& +& ...& +& 
a_{tt}w_{n+t}
\end{array}$$
\end{proposition}
\begin{proof}
Consider the following  $(t+1) \times (n+t)$ matrix $B$:
$$B = \left( \begin{array}{ccccccccc}1 & 1 & ...& 1 & ... & 1 & 1  & ... & 1\\ 
a_{11}-1 & a_{12} & ... & a_{1t} & ... & a_{1n} & a_{11} &   ... & a_{1t} \\
a_{21} & a_{22}-1 & ... & a_{2t} &... & a_{2n} & a_{21} &   ... & a_{2t} \\
\vdots &  &  & & & &&  &\\
a_{t1} &  a_{t2} & ... & a_{tt}-1 &... & a_{tn} & a_{t1} &    ...  &a_{tt}\\
\end{array} \right)$$

Then the existence of rationals $w_1, w_2, ... , w_{n+t}$ which satisfy the indicated equations is equivalent to the existence of a solution in $\mathbb{Q}^{n+t}$  to the system of $t+1$ equations represented by the vector equation $B\vec{x} = (1,0,0,...,0)^t$.  

We claim that the  $t+1$  rows of $B$ are linearly independent. 
For $1 \leq i \leq t$ we subtract column $i$ from column $n+i$ in the matrix $B$, 
which yields the column-equivalent matrix $C$: 
$$C = \left( \begin{array}{cccccccccc}1 & 1 & ...& 1 & ... & 1 & 0  & 0 &... & 0\\ 
a_{11}-1 & a_{12} & ... & a_{1t} & ... & a_{1n} & 1 & 0&  ... & 0 \\
a_{21} & a_{22}-1 & ... & a_{2t} &... & a_{2n} & 0 & 1&  ... & 0 \\
\vdots &  &  & & & &&  &\\
a_{t1} &  a_{t2} & ... & a_{tt}-1 &... & a_{tn} & 0 &  0 &  ...  &1\\
\end{array} \right)$$  
The final $t+1$ columns of $C$ are clearly linearly independent, so that ${\rm columnrank}(C) \geq t+1$.   But ${\rm columnrank}(C) = {\rm columnrank}(B) = {\rm rowrank}(B) \leq t+1$ (since $B$ has $t+1$ rows), so that ${\rm columnrank}(C) =  {\rm rowrank}(B) = t+1$.

Since the $t+1$ rows of $B$ are linearly independent in $\mathbb{Q}^{n+t}$, there exists a solution to any system of equations of the form $B\vec{x} = \vec{c}$ for any $\vec{c}$ in $\mathbb{Q}^{t+1}$, in particular for $\vec{c} =  (1,0,0,...,0)^t$.
\end{proof}

\smallskip

We now have in place all the tools we need to establish our main result, that any Cohn path algebra $C_K(E)$ has IBN (Theorem \ref{MainResult}).   But prior to doing so, we make an observation which will provide some context.   By Theorem \ref{everyCohnisLeavitt}, together with the specific construction of the graph $F = F(E)$, we see that $C_K(E)$ is isomorphic to a Leavitt path algebra $L_K(F)$ for which the graph $F$ necessarily has sinks.   So one might be tempted to conjecture that our main result is simply an artifact of a result about Leavitt path algebras of graphs with sinks.  But indeed there are numerous examples of graphs $G$ with sinks for which the Leavitt path algebra $L_K(G)$ does not have IBN; for instance the graph $G = E(X)$ described  Example \ref{relativeCohnExample} below is such.  

Here now is our main result.  

\begin{theorem} \label{MainResult}
Let $E$ be any finite graph, and $K$ any field.  Then the Cohn path algebra $C_K(E)$ has the Invariant Basis Number property. 
\end{theorem}

{\bf Proof.} By Theorem \ref{everyCohnisLeavitt} we have $C_K(E) \cong L_K(F)$, where $F = F(E)$ is the graph described in Definition \ref{FgraphDef}.  Specifically, if $E^0 = \{v_1, v_2, ... , v_n\}$, and we label the regular vertices of $E$ as $v_1, ..., v_t$, then $F^0 = \{v_1, v_2, ..., v_n, v_1', v_2', ..., v_t'\}$, and the only regular vertices of $F$ are $\{v_1, v_2, ..., v_t\}$.  Note that  $|F^0| = n+t$.    

Let $A = A_E = (a_{i,j})$ be the incidence matrix of $E$.  Then  the  monoid $M_F$ is the monoid $(\Z^+)^{n+t}$,  modulo the equivalence relation generated by setting 
\begin{equation}
\vec{b_i} \sim_F (a_{i,1}, a_{i,2}, ..., a_{i,n}, a_{i,1}, a_{i,2}, ... , a_{i,t})  \mbox{  for each } 1\leq i \leq t.  \tag{$\dagger$}  	
\end{equation}
 
Let $\vec{\rho}$ denote the element 
$ (1,1,...,1)$ of $(\Z^+)^{n+t}$.   We establish, for any pair of positive integers $m \neq m'$, that $m\vec{\rho} \nsim_F m'\vec{\rho}$.

Let $w_1, w_2, ..., w_{n+t} \in \mathbb{Q}$ be rationals which satisfy the linear system presented in, and whose existence is guaranteed by,   Proposition \ref{linearalgprop}.  In particular, for $1\leq i \leq t$ we have $w_i = \sum_{\ell=1}^{n+t} a_{i,\ell} w_\ell$.      We define the map $\Gamma: (\Z^+)^{n+t} \rightarrow \mathbb{Q}$ by setting $$\Gamma ((z_1, z_2, ..., z_{n+t})) =  \sum_{\ell=1}^{n+t} z_\ell w_\ell.$$   Then $\Gamma$ is clearly linear.  We claim that if $(z_1, z_2, ..., z_{n+t}) \sim_F (z_1', z_2', ..., z_{n+t}')$, then \\
$\Gamma( (z_1, z_2, ..., z_{n+t}) )= \Gamma ((z_1', z_2', ..., z_{n+t}'))$.    The equivalence relation $\sim_F$ is generated by the relations given in 
($\dagger$), so it suffices to show that the claim holds if we apply any one of these generating relations to $(z_1, z_2, ..., z_{n+t})$ to produce $ (z_1', z_2', ..., z_{n+t}')$.   That is, it suffices to establish the claim in the situation where 
$(z_1, z_2, ..., z_{n+t}) = \vec{a} + \vec{b_i}$ and $ (z_1', z_2', ..., z_{n+t}') = \vec{a} + (a_{i,1}, a_{i,2}, ... a_{i,n}, a_{i,1}, a_{i,2}, ... a_{i, t})$,  for each $1\leq i \leq t$, and for any $\vec{a} \in (\Z^+)^{n+t}$.  Using the definition and linearity of $\Gamma$, we get
\begin{align}
\Gamma ((z_1, z_2, ..., z_{n+t})) &= \Gamma( \vec{a} + \vec{b_i}) = \Gamma(\vec{a}) + \Gamma(\vec{b_i}) = \Gamma(\vec{a}) + 1w_i = \Gamma(\vec{a}) + \sum_{\ell = 1}^{n+t} a_{i,\ell}w_{\ell} \notag \\
& =  \Gamma(\vec{a}) + \Gamma((a_{i,1}, a_{i,2}, ... a_{i,n}, a_{i,1}, a_{i,2}, ... a_{i, t})) 
\notag \\ &= \Gamma((z_1', z_2', ..., z_{n+t}') ). \notag
\end{align}
\noindent
Thus we have established the claim. 

But recall that, by Proposition \ref{linearalgprop},  the $w_{\ell}$ have been chosen so that \\ $\sum_{\ell = 1}^{n+t} w_{\ell} = 1$.  So in particular for any positive integer $m$  we get
$$\Gamma(m\vec{\rho}) = \Gamma((m,m,...,m)) = \sum_{\ell = 1}^{n+t} mw_{\ell} = m \sum_{\ell = 1}^{n+t} w_{\ell} = m \cdot 1 = m.$$    So for $m\neq m'$ we have $\Gamma(m\vec{\rho}) = m \neq m' =  \Gamma(m'\vec{\rho})$, so that  by the contrapositive of the claim we conclude that if $m\neq m'$, then  $m\vec{\rho} \nsim_F m'\vec{\rho}$. Now Corollary \ref{showingL(E)IBN} completes the proof. 
\hfill $\Box$

As one consequence of our main result, we state that any Cohn path algebra over a finite graph will also have Invariant Matrix Number. This result follows from the next proposition that we recall from \cite{AS}.  
\begin{proposition}(\cite[Proposition 4]{AS})\label{L(E)hasIMN} If $F$ is a finite graph, and $[L_K(F)]$ has infinite order in $K_0(L_K(F))$, then $L_K(F)$ has Invariant Matrix Number.
\end{proposition}

\begin{corollary} \label{CortoMainResult}
Let $E$ be any finite graph, and $K$ any field. Then the Cohn path algebra $C_K(E)$ has Invariant Matrix Number.
\end{corollary}

{\bf Proof.} Using the same terminology and setup as in the proof of the main theorem, we will show that $C_K(E) \cong L_K(F)$ has Invariant Matrix Number.  

Let $R$ be any unital ring, and suppose that $R$ has finite order in $K_0(R)$.  Then $n[R] = [0]$ in 
$K_0(R)$ for some $n\in \mathbb{N}$.  By the construction of $K_0$, this means that $R^n$ is stably isomorphic to $\{0\}$; i.e.,  that there exists $m\in \mathbb{N}$ with $R^n \oplus R^m \cong \{0\} \oplus R^m$ as left $R$-modules.  But then $R^{n+m} \cong R^m$ as left $R$-modules.   The upshot of this discussion is that for any unital ring $R$ having Invariant Basis Number, then $[R]$ has infinite order in $K_0(R)$.

Thus Theorem 2.2 together with Proposition 2.2 apply to yield the result. 
\hfill $\Box$

\smallskip

The Cohn path algebra $C_K(E)$ and the Leavitt path algebra $L_K(E)$ of a graph $E$ may be thought of as occupying two ends of a spectrum:  to build $L_K(E)$ we impose the (CK2) relation at {\it all} of the vertices of $E$, whereas to build $C_K(E)$ we impose the (CK2) relation at {\it none} of the vertices.  There is intermediate ground: for any set $X \subseteq {\rm Reg}(E)$ we define  the {\it relative} Cohn path algebra $C_K^X(E)$ to be the $K$-algebra generated by the same relations as the relations which generate $C_K(E)$, where in addition we impose  the (CK2) relation precisely at the vertices in $X$. 
(See \cite[Section 5]{AAS} for additional information.) So $C_K(E) = C_K^\emptyset (E)$, while $L_K(E) = C_K^{{\rm Reg}(E)}(E)$.  Cast in the context of Theorem \ref{MainResult}, the case $X = \emptyset$ (i.e, $|X| = 0$) plays a unique role, in this sense:  for any positive integers $1 \leq m \leq n$ there exists a  graph $E$ and $X \subseteq E^0$ having $|E^0| = n$, $|X|=m$, and for which  $C^X_K(E)$ does not have IBN.    We describe here such a graph in the case $m=1, n=2$; the construction easily generalizes to arbitrary $m,n$.   

\begin{example}\label{relativeCohnExample}    
{\rm 
Let $E$ be the graph  
$  \xymatrix{\bullet^v \ar@(u,l)^{} &  \bullet^w 
\ar[l]^{} \ar@(u,r)^{} \ar@(r,d)^{}   }$, and let $X = \{w \}$.  By \cite[Theorem 1.5.17]{AAS}, $C^X_K(E) \cong L_K(E(X))$, 
where the graph $E(X)$ is 
$  \xymatrix{\bullet^v \ar[d]\ar@(u,l)^{} &  \bullet^w 
\ar[l]^{} \ar@(u,r)^{} \ar@(r,d)^{} 
\ar[dl]_{}\\
\bullet^{v'} & 
  }$.
Here  $M_{E(X)}$ is $(\Z^+)^3$ modulo the equivalence relation $\sim_{M(E)}$ generated by setting 
 $(1,0,0) \sim_{E(X)} (1,0,1)$ and $(0,1,0) \sim_{E(X)} (1,2,1)$.  But then $ \vec{\rho} = (1,1,1) = (1,0,1) + (0,1,0) \sim_{E(X)}  (1,0,1) + (1,2,1) = (2,2,2) = 2\vec{\rho}$, so that 
$L_K(E(X))$, and thus also $C_K^X(E)$, does not have the IBN property by 
Corollary~\ref{showingL(E)IBN}.
For the general case,  
Let $E_n$ be the graph 
$$  
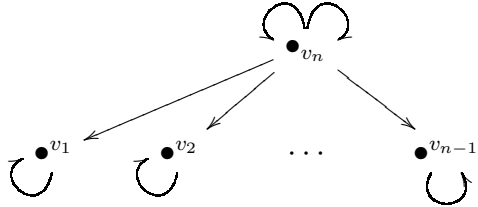
\xymatrix{ & &  \bullet_{v_n} \ar[dll]_{} \ar[dl]_{} \ar[dr]_{}\ar@(u,l)^{} \ar@(u,r)^{}  & \\
\bullet^{v_1} \ar@(d,l)^{} & \bullet^{v_2} \ar@(d,l)^{} &\cdots & \bullet^{v_{n-1}} 
\ar@(dl,dr)^{}}$$ 
and $X_m = \{v_{n-m+1},...,v_n \}$.   Then using the same ideas as for the graph $E$ above, it is easy to establish that $C_K^{X_m}(E_n)$ does not satisfy the IBN property.   
\hfill $\Box$}
\end{example}

We want to point out that, in $L_K(F)$ with IBN property, for $m\neq m'$, although $ m\vec{\rho} \not\sim_F m'\vec{\rho}$, there can exist elements  $x$  in $M_F$ for which $mx \sim_F m'x$. 
Take for instance, the graph $F=F(R_2)$ of Example~\ref{F(R_2)Example} and 
consider the element  $x = (1,2)$. Then $x= (1,2) \sim_F (2,2)+(0,2) = (2,4)= 2x $.

We conclude this short note by providing an alternate proof of the main result.   Unlike the approach used above (an `invariant-based' method), this alternate approach makes use of an explicit description of $\mathcal{V}(C_K(E))$, together with a result about monoids generated by relations of a specified type. This alternate approach was suggested by P. Ara, to whom the authors are extremely grateful.

\medskip

{\bf Theorem \ref{MainResult}, re-established}.    We define the monoid $M_{C(E)}$  as follows.  Let $T$ denote the free abelian monoid (written additively) with generators   $E^0 \ \sqcup \  \{q_v \  | \ v \ \mbox{is regular}\}$.   Define relations in $T$  by setting 
\begin{equation}
v =  q_v + \sum_{e\in s^{-1}(v)}r(e)
\tag{$\ddagger$}
\end{equation}
 for each $v \in {\rm Reg}(E)$.   Let $\sim$ be the equivalence relation in $T$ generated by these relations.  Then $M_{C(E)} = T / \sim$, with operation defined by by setting $[x] + [y] = [x+y]$.

By \cite[Theorem 4.3]{AG}, $\mathcal{V}(C_K(E)) \cong M_{C(E)}$; moreover, under this isomorphism, $[C_K(E)] \mapsto [\sum_{v \in E^0}v]$.   We denote $E^0$ by $\{v_1, v_2, ..., v_n\}$
.   Thus we achieve the desired conclusion by showing that, for positive integers $m,m'$, if  $m[\sum_{i=1}^n v_i ] = m'[\sum_{i=1}^n v_i ]$ in $M_{C(E)}$, then $m=m'$.     

 For each $t\in T$ Let  $\ddagger_i(t)$ denote the element of $T$ which results by applying the relation ($\ddagger$) corresponding to vertex $v_i$ on $t$.   For any sequence $\sigma $ (of any length) taken from $\{1,2,...,n\}$, and any $t\in T$,  let $\Lambda_{\sigma}(t) \in T$ be the element which results by applying $\ddagger_i$ operations in the order specified by $\sigma$.   
 
 The monoid $M_{C(E)}$ is given by generators and relations of the form described in \cite[Section 4]{AMP} (where the monoid is denoted by $F$).  In particular, by \cite[Lemma 4.3]{AMP} (the so-called `Confluence Lemma'), the hypothesis $m[\sum_{i=1}^n v_i ] = m'[\sum_{i=1}^n v_i ]$ in $M_{C(E)}$ yields that there are  two sequences $\sigma, \sigma'$ for which  
 $$\Lambda_{\sigma}(m\sum_{i=1}^n v_i ) = \gamma = \Lambda_{\sigma'}(m'\sum_{i=1}^n v_i )$$ 
 for some $\gamma \in T$. 
But each time a substitution of the form $\ddagger_i$  is made to an element of $T$, the effect on that element  is to:   
\begin{enumerate}
\item subtract $1$ from the coefficient on $v_i$;  
\item add $a_{ij}$ to the coefficient on $v_j$ (for $1\leq j \leq n$); and 
\item add $1$ to the coefficient on $q_{v_i}$. 
\end{enumerate}

As one consequence of this observation we see that, starting with $t \in T$ for which the coefficient on each $q_{v_i}$ in $t$ is $0$, the coefficient on $q_{v_i}$ in $\Lambda_{\sigma}(t)$ is precisely the number of times the relation $\ddagger_i$ was invoked.   In particular, since both $m\sum_{i=1}^n v_i $ and $m' \sum_{i=1}^n v_i $ have this property,  the equation $\Lambda_{\sigma}(m\sum_{i=1}^n v_i )  = \Lambda_{\sigma'}(m'\sum_{i=1}^n v_i )$ in $T$ yields that for each regular vertex $v_i$, the number of times $\ddagger_i$ is invoked in $\Lambda_{\sigma}$ is the same as the number of times that $\ddagger_i$ is invoked in $\Lambda_{\sigma'}$.    Denote this common number by $k_i$.    Recalling the previously observed effect of $\ddagger_i$ on any element of $T$, we see that 
\begin{align}
\gamma = &\Lambda_{\sigma}(m\sum_{i=1}^n v_i ) \notag \\
 = &((m - k_1)+ k_1a_{11}+k_2a_{21}+...+k_na_{n1})v_1 \notag \\
 &+((m-k_2)+k_1a_{12}+k_2a_{22}+...+k_na_{n2})v_2 + \cdots 
\notag \\
&  +
((m - k_n) +k_1a_{1n}+k_2a_{2n}+...+k_na_{nn})v_n + k_1 q_{v_1} + k_2 q_{v_2} + \cdots + k_n q_{v_n} .\notag 
\end{align}
But then also
\begin{align}
\gamma = &\Lambda_{\sigma'}(m'\sum_{i=1}^n v_i ) \notag \\
= &((m' - k_1) + k_1a_{11}+k_2a_{21}+...+k_na_{n1})v_1 \notag \\
&+((m'-k_2)+k_1a_{12}+k_2a_{22}+...+k_na_{n2})v_2 + \cdots
\notag \\
& +
((m' - k_n) +k_1a_{1n}+k_2a_{2n}+...+k_na_{nn})v_n +  k_1 q_{v_1} + k_2 q_{v_2} + \cdots + k_n q_{v_n} .
\notag 
\end{align}

By equating coefficients on any of the free generators  $v_i$ of $T$,  we conclude that $m-k_i = m' - k_i$, so that $m = m'$ as desired.

\end{document}